\documentclass[12pt]{amsart}
\usepackage{graphicx}
\usepackage{amssymb}
\newtheorem{thm}{Theorem}[section]
\newtheorem{cor}[thm]{Corollary}

\newtheorem{prop}[thm]{Proposition}
\theoremstyle{definition}
\newtheorem{defn}[thm]{Definition}
\theoremstyle{remark}
\newtheorem{rem}[thm]{Remark}
\numberwithin{equation}{section}
\newtheorem{exa}[thm]{Example}

\newcommand{\h}{\mathcal{H}}

\newcommand{\K}{\mathcal{K}}

\DeclareMathOperator{\Ran}{ran}
\DeclareMathOperator{\Span}{span}

\begin{document}
\title[$G$-frame representations]{$G$-frame representations with bounded operators}%
\author {Yavar Khedmati and Fatemeh Ghobadzadeh}%
\address{
\newline
\indent Department of Mathematics
\newline
\indent Faculty of Sciences
\newline \indent   University of Mohaghegh Ardabili
\newline \indent  Ardabil 56199-11367
\newline \indent Iran}
\email{khedmati.y@uma.ac.ir, khedmatiy.y@gmail.com}
\email{gobadzadehf@yahoo.com}

\begin{abstract}
Dynamical sampling, as introduced by Aldroubi et al., deals with frame properties of sequences of the form $\{T^i f_1\}_{i\in \mathbb{N}}$, where $f_1$ belongs to Hilbert space $\h$ and $T:\h\rightarrow\h$ belongs to certain classes of the bounded operators. Christensen et al., study frames for $\h$ with index set $\mathbb{N}$ (or $\mathbb{Z}$), that have representations in the form $\{T^{i-1}f_1\}_{i\in \mathbb{N}}$ (or $\{T^if_0\}_{i\in \mathbb{Z}}$).
As frames of subspaces, fusion frames and generalized translation invariant systems are the spacial cases of $g$-frames, the purpose of this paper is to study $g$-frames $\Lambda=\{\Lambda_i\in B(\h,\K): i\in I\}$ $(I=\mathbb{N}$ or $\mathbb{Z}$)  having the form $\Lambda_{i+1}=\Lambda_1 T^{i},$ for $T\in B(\h).$ 
\end{abstract}
\maketitle
\section{{\textbf Introduction}}
In 1952, the concept of frames for Hilbert spaces was defined by Duffin and Schaeffer \cite{DS}. Frames are important tools in the signal/image processing \cite{Anto, Balador, Desti} , data compression \cite{Data2, Data1}, dynamical sampling \cite{ald1, ald2} and etc. 
\par Throughout this paper, $I$ is a countable index set, $\h$ and $\K$ are seperable Hilbert spaces, $\{\K_i: i\in I\}$ is a family of seperable Hilbert spaces, $Id_{\h}$ denotes the identity operator on $\h$,  $B(\h)$ and $GL(\h)$ denote the set of all bounded linear operators and the set of all invertible bounded linear operators on $\h$, respectively, $l^2(\h,I)=\big\{\{g_i\}_{i\in I}:g_i\in\K,\sum_{i\in\mathbb{Z}}\|g_i\|^2<\infty\big\}$ as well. Also, we will apply $B(\h,\K)$ for the set of all bounded linear operators from $\h$ to $\K$. We use $\ker T$ and $\Ran T$ for the null space and range $T\in B(\h)$, respectively. 
We denote the natural, integer and complex numbers by $\mathbb{N}$, $\mathbb{Z}$ and $\mathbb{C}$, respectively.
\par A sequence $F=\{f_i\}_{i\in I}$ in $\h$ is called a  frame for $\h$ if there exist two constants $A_F, B_F> 0$ such that
\begin{eqnarray}\label{abc}
A_F \|f\|^2\leq\sum_{i\in I}|\langle f,f_i\rangle |^2 \leq B_F \|f\|^2,\quad f\in \h.
\end{eqnarray}
Let $F=\{f_i\}_{i\in I}$ be a frame for $\h,$ then the operator
\begin{eqnarray*}
T_F:l^2(\h,I) \rightarrow\h,\quad T_F(\{c_i\}_{i\in I})=\sum_{i\in I} c_if_i,
\end{eqnarray*}
is well define and onto, also its adjoint is
\begin{eqnarray*}
T^*_F:\h\rightarrow l^2(\h,I) ,\quad T_F^* f=\{\langle f,f_i \rangle\}_{i\in I}.
\end{eqnarray*}
The operators $T_F$ and $T^*_F$ are called the synthesis and analysis operators of frame $F,$ respectively.
\par Frames for $\h$ allow each $f \in\h$ to be expanded as an (infinite) linear combination of the frame elements.
For more on frame we refer to the book \cite{c4}.
Aldroubi et al., introduced the concept of dynamical sampling deals with frame properties of sequences of the form $\{T^i f_1\}_{i\in \mathbb{N}}$, for $f_1\in\h$ and $T:\h\rightarrow\h$ belongs to certain classes of the bounded operators.
Christensen and Hassannasab analyze frames $F=\{f_i\}_{i\in\mathbb{Z}}$ having the form $F=\{T^{i} f_0\}_{i\in\mathbb{Z}}$, where $T$ is a bijective linear operator (not necessarily bounded) on $\Span\{f_i\}_{i\in\mathbb{Z}}$. They show, $(T^*)^{-1}$ is the only possibility of the representing operator for the duals of the frame $F=\{f_i\}_{i\in\mathbb{Z}}=\{T^{i}f_0\}_{i\in\mathbb{Z}}$, $T\in GL(\h)$ \cite{main}. They even clarify stability of the representation of frames. Christensen et al., determine the frames that have a representation with a bounded operator, and survey the properties of this operators \cite{chretal}.
\begin{prop}\cite{main}\label{mnref}
Consider a frame sequence $F=\{f_i\}_{i\in \mathbb{Z}}$ in $\h$ which spans an infinite dimentional subspace.
The following are equivalent:
\begin{enumerate}
\item[(i)] $F$ is linearly independent. 
\item[(ii)] The map $Tf_i:=f_{i+1}$ is well-defined and extends to a linear and invertible operator 
$T:\Span\{f_i\}_{i\in \mathbb{Z}}\rightarrow\Span\{f_i\}_{i\in \mathbb{Z}}$.
\end{enumerate}
In the affirmative case, $F=\{T^i f_0\}_{i\in \mathbb{Z}}$.
\end{prop}
\begin{thm}\cite{chretal}
Consider a frame $F=\{f_i\}_{i\in\mathbb{N}}$ in $\h$. Then the following are equivalent:
\begin{enumerate}
\item[(i)] $F$ has a representation $F=\{T^{i-1} f_1\}_{i\in\mathbb{N}}$ for some $T\in B(\h)$.
\item[(ii)] For some dual frame $G=\{g_i\}_{i\in\mathbb{N}}$ (and hence all)
\begin{align*}
f_{j+1}=\sum_{i\in\mathbb{N}}\langle f_j,g_i\rangle f_{i+1},\forall j\in\mathbb{N}.
\end{align*}
\item[(iii)] The $\ker T_F$ is invariant under the right-shift operator.
\end{enumerate}
In the affirmative case, letting $G=\{g_i\}_{i\in\mathbb{N}}$ denote an arbitrary dual frame of $F$, the operator
 $T$ has the form 
 \begin{align*}
Tf=\sum_{i\in\mathbb{N}}\langle f,g_i\rangle f_{i+1},\forall f\in\h,
\end{align*}
and $1\leq\|T\|\leq\sqrt{B_FA_F^{-1}}$.
\end{thm}
 We will need the right-shift operator on $l^2(\h,\mathbb{N})$ and $l^2(\h,\mathbb{Z})$, defined by $\mathcal{T}(\{c_i\}_{i\in\mathbb{N}})=(0,c_1,c_2,...)$ and $\mathcal{T}(\{c_i\}_{i\in\mathbb{Z}})=\{c_{i-1}\}_{i\in\mathbb{Z}}$. 
 Clearly, the right-shift operator on $l^2(\h,\mathbb{Z})$ is unitary and $\mathcal{T}^*$ is the left-shift operator, i.e. $\mathcal{T}^*(\{c_i\}_{i\in\mathbb{Z}})=\{c_{i+1}\}_{i\in\mathbb{Z}}$. A subspace $V\subseteq l^2(\h,\mathbb{N})$ is invariant under right-shift (left-shift) if $\mathcal{T}(V)\subseteq V$ ($\mathcal{T}^*(V)\subseteq V$).
 \par In 2006, generalized frames (or simply $g$-frames) and $g$-Riesz bases were introduced by Sun \cite{ws}. "$G$-frames are natural generalizations of frames which cover many other recent generalizations of frames, e.g., bounded quasi-projectors, frames of subspaces, outer frames, oblique frames, pseudo-frames and a class of time-frequency localization operators \cite{wsper}. The interest in $g$-frames arises from the fact that they provide more choices on analyzing functions than frame expansion coefficients \cite{ws}, and also, every fusion frame is a $g$-frame \cite{c4}." Generalized translation invariant (GTI) frames can be realized as $g$-frames \cite{dtuykhmj}, so for motivating to answer the similar problems relevant to shift invariant systames and GTI systems in \cite{oprobs}, we generalize some results of the frame representations with bounded operators in \cite{main, chretal} to $g$-frames. Now, we summarize some facts about $g$-frames from \cite{Anaj and Rah, ws}.
 \begin{defn}
We say that $\Lambda=\{\Lambda_{i}\in B(\h,\K_{i}):i\in I\}$ is a generalized frame, or simply $g$-frame, for $\h$  with respect to $\{\K_{i}:i\in I\}$ if there are two constants $0<A_\Lambda\leq B_\Lambda<\infty$ such that
\begin{equation}\label{cgframe}
A_\Lambda\|f\|^{2}\leq\sum_{i\in I}\|\Lambda_i f\|^{2}\leq B_\Lambda\|f\|^{2},\; f\in \h.
\end{equation}
We call $A_\Lambda,B_\Lambda$ the lower and upper $g$-frame bounds, respectively.
$\Lambda$ is called a tight $g$-frame if $A_\Lambda=B_\Lambda,$ and a
Parseval $g$-frame if $A_\Lambda=B_\Lambda=1.$ 
If for each $i\in I,$ $\K_{i}=\K,$ then, $\Lambda$ is called a $g$-frame for $\h$ with respect to $\K$. 
Note that for a family $\{\K_i\}_{i\in I}$ of Hilbert spaces, there exists a Hilbert space $\K=\oplus_{i\in I}\K_i$ such that for all $i\in I$, $\K_i\subseteq\K$, where $\oplus_{i\in I}\K_i$ is the direct sum of $\{\K_i\}_{i\in I}$.
A family $\Lambda$ is called a $g$-Bessel family for $\h$  with respect to $\{\K_{i}:i\in I\}$ if the right hand inequality in (\ref{cgframe}) holds for all $f\in \h,$ in this case, $B_\Lambda$ is called the $g$-Bessel bound.
\end{defn}
 If there is no confusion, we use $g$-frame ($g$-Bessel family) instead of $g$-frame for $\h$ with respect to $\{\K_{i}:i\in I\}$ ($g$-Bessel family for $\h$ with respect to $\{\K_{i}:i\in I\}$). 
 \begin{exa}\cite{ws}
 Let $\{f_i\}_{i\in I}$ be a frame for $\h$. Suppose that 
 $\Lambda=\{\Lambda_{i}\in B(\h,\mathbb{C}):i\in I\}$, where  $$\Lambda_i f=\langle f,f_i\rangle,\quad f\in\h.$$ 
 It is easy to see that $\Lambda$ is a $g$-frame.
 \end{exa}
For a $g$-frame $\Lambda$, there exists a unique positive and invertible operator $S_\Lambda:\h\rightarrow \h$ such that
\begin{eqnarray*}\label{di}
S_\Lambda f=\sum_{i\in I}\Lambda_i^*\Lambda_i f,\quad f\in\h,
\end{eqnarray*}
 and $A_\Lambda Id_{\h}\leq S_\Lambda\leq B_\Lambda Id_{\h}.$ Consider the space
$$\Big(\sum_{i\in I}\oplus\K_i\Big)_{l^2}=\Big\{\{g_i\}_{i\in I}: g_i\in \K_i,\:i\in I\: and\: \sum_{i\in I}\|g_i\|^2<\infty\Big\}.$$
It is clear that, $\Big(\sum_{i\in I}\oplus\K_i\Big)_{l^2}$ is a Hilbert space with point wise operations and with the innerproduct given by $$\big\langle \{f_i\}_{i\in I},\{g_i\}_{i\in I}\big\rangle=\sum_{i\in I}\langle f_i,g_i\rangle.$$
For a $g$-Bessl $\Lambda$, the synthesis operator $T_\Lambda:\Big(\sum_{i\in I}\oplus\K_i\Big)_{l^2}\rightarrow\h$ is defined by 
\begin{align*}
T_\Lambda\big(\{g_i\}_{i\in I}\big)=\sum_{i\in I}\Lambda_i^*g_i.
\end{align*}
The adjoint of $T_\Lambda$, $T_\Lambda^*:\h\rightarrow\Big(\sum_{i\in I}\oplus\K_i\Big)_{l^2}$ is called the analysis operator of $\Lambda$ and is as follow
\begin{align*}
T_\Lambda^* f=\{\Lambda_i f\}_{i\in I},\quad f\in\h.
\end{align*}
It is obvious that $S_\Lambda=T_\Lambda T_\Lambda^*.$
\begin{defn}
Two $g$-frames $\Lambda$ and $\Theta$ are called dual if
\begin{align*}
\sum_{i\in I}\Lambda_i^*\Theta_i f=f,\quad f\in\h.
\end{align*}
\end{defn}
For a $g$-frame $\Lambda=\{\Lambda_{i}\in B(\h,\K_i):i\in I\}$, the $g$-frame
$\widetilde{\Lambda}=\{\Lambda_i S_\Lambda^{-1}\in B(\h,\K_i):i\in I\}$ is a dual of $\Lambda$, that is called canonical dual.
 \begin{defn} Consider a family $\Lambda=\{\Lambda_{i}\in B(\h,\K_i):i\in I\}$.
 \begin{enumerate}
 \item[(i)] 
 We say that $\Lambda$ is $g$-complete if $\{f: \Lambda_if=0, i\in I\}=\{0\}$.
 \item[(ii)]
 We say that $\Lambda$ is a $g$-Riesz basis if $\Lambda$ is $g$-complete and there are two constants $0<A_\Lambda\leq B_\Lambda<\infty$ such that for any finite set $I_n\subset I$ 
\begin{align*}
A_\Lambda\sum_{i\in I_n}\|g_i\|^{2}\leq\|\sum_{i\in I_n}\Lambda_i^*g_i\|^{2}\leq B_\Lambda\sum_{i\in I_n}\|g_i\|^{2},\; g_i\in \K_i.
\end{align*}
\item[(iii)] We say that $\Lambda$ is a $g$-orthonormal basis if it satisfies the following:
\begin{align*}
&\langle\Lambda_i^*g_i,\Lambda_j^* g_j\rangle=\delta_{i,j}\langle g_i,g_j\rangle,\quad i,j\in I, g_i\in\K_i, g_j\in \K_j,
\\&\sum_{i\in I}\|\Lambda_i f\|^2=\|f\|^2,\quad f\in\h.
\end{align*}
\end{enumerate}
\end{defn}
\begin{thm}\label{reisoth}
A family $\Lambda=\{\Lambda_{i}\in B(\h,\K_i):i\in I\}$ is a $g$-Riesz basis if and only if there exist a $g$-orthonormal basis $\Theta$ and $U\in GL(\h)$ such that 
$\Lambda_i=\Theta_i U,\:i\in I.$
\end{thm}
\begin{thm}\cite{ws}\label{framegframe}
 Let for $i\in I$, $\{e_{i,j}\}_{j\in J_i}$ be an orthonormal basis for $\K_i$,
\begin{enumerate}
\item[(i)] $\Lambda$ is a $g$-frame (respectively, $g$-Bessel family, $g$-Riesz basis, $g$-orthonormal basis) if and only if 
$\{\Lambda_i^*e_{i,j}\}_{i\in I, j\in J_i}$ is a frame (respectively, Bessel sequence, Riesz basis, orthonormal basis).
\item[(ii)] $\Lambda$ and $\Theta$ are dual if and only if $\{\Lambda_i^*e_{i,j}\}_{i\in I, j\in J_i}$ and $\{\Theta_i^*e_{i,j}\}_{i\in I, j\in J_i}$ are dual.
\end{enumerate}
\end{thm}
In this paper, we generalize some recent results of Christensen et al., \cite{main, chretal} to investigate representations for $g$-frames with bounded operators.
\section{{\textbf Representations of $g$-frames}}
In this section, by generalizing some results of \cite{main, chretal}, we introduce representation for $g$-frames with bounded operators and give some examples of $g$-frames with a representation and without any representations. In the Theorem \ref{MT}, we get sufficient conditions for $g$-frames to have a representation with bounded operator. Also, the Theorem \ref{MT} and the Peoposition \ref{inv} show that for $g$-frames $\Lambda=\{\Lambda_1T^{i-1}:i\in\mathbb{N}\}$, the boundedness of $T$ is equivalent to the invariant of the $\ker T_\Lambda$ under right-shift.
\begin{rem}\label{frem}
Consider a frame
$F=\{f_i\}_{i\in{\mathbb{N}}}=\{T^{i-1}f_1\}_{i\in{\mathbb{N}}}$ for
$\h$ with $T\in{B(\h)}$. For the $g$-frame $\Lambda=\{\Lambda_{i}\in
{B(\h,\mathbb{C})}: i\in{\mathbb{N}}\}$ where 
\begin{align*}
\Lambda_i f=\langle f,f_i\rangle,\quad f\in\h,
\end{align*}
we have
\begin{align*}
\Lambda_{i+1}f=\langle f,f_{i+1}\rangle=\langle
f,Tf_{i}\rangle=\langle T^*f,f_{i}\rangle=\Lambda_{i}T^*f,\quad
f\in{\h}.
\end{align*}
Therefore, there exists $S\in{B(\h)}$ such that $\Lambda_i=\Lambda_1
S^{i-1}, i\in{\mathbb{N}}$. Conversely, if
$\Lambda=\{\Lambda_i\in{B(\h,\mathbb{C})}:
i\in{\mathbb{N}}\}=\{\Lambda_{1}T^{i-1}:i\in{\mathbb{N}}\}$ for
$T\in{B(\h)}$, then
$F=\{f_{i}\}_{i\in{\mathbb{N}}}=\{(T^*)^{i-1}f_1\}_{i\in{\mathbb{N}}}$,
where $\Lambda_{i}f=\langle f,f_{i}\rangle, i\in{\mathbb{N}},
f\in{\h}$.
\end{rem}
We are motivated to study $g$-frames $\Lambda=\{\Lambda_i\in B(\h,\K):i\in\mathbb{N}\},$ where
$\Lambda_i=\Lambda_1 T^{i-1}$ with $T\in{B(\h)}$.
\begin{defn}
We say that a $g$-frame $\Lambda=\{\Lambda_i\in B(\h,\K):i\in\mathbb{N}\}$ has a representation, if there is a
$T\in{B(\h)}$ such that $\Lambda_{i}=\Lambda_1 T^{i-1},
i\in{\mathbb{N}}$. In the affrimative case, we say that $\Lambda$ is
represented by $T$.
\end{defn}
In the following, we give some $g$-frames that have a representation.
\begin{exa}\label{tless1}
\begin{enumerate}
\item[(i)] The $g$-frame $\Lambda=\{\Lambda_{i}\in{GL(\h)}: i=1,2\}$
is represented by $\Lambda_1^{-1}\Lambda_2$.
\item[(ii)] The tight $g$-frame $\Lambda=\{\Lambda_{i}\in{B(\h)}:
i\in{\mathbb{N}}\}$ with
$\Lambda_{i}=\frac{2^{i-1}}{3^{i-2}}Id_{\h}$ is represented by
$\frac{2}{3}Id_{\h}$.
\item[(iii)] Let
$F=\{f_{i}\}_{i\in{\mathbb{N}}}=\{T^{i-1}f_1\}_{i\in{\mathbb{N}}}$
be a frame for $\h$. Then the $g$-frame
$\Lambda=\{\Lambda_{i}\in{B(\h,\mathbb{C}^2)}: i\in{\mathbb{N}}\}$
with $\Lambda_{i}f=\big(\langle f,f_{i}\rangle,\langle
f,f_{i+1}\rangle\big)$, \\$f\in{\h}$, is represented by $T^*$.
\end{enumerate}
\end{exa}
Now, we give a $g$-frame without any representations.
\begin{exa}
Consider the tight $g$-frame
$\Lambda=\{\Lambda_{n}\in{B(\mathbb{C})}: n\in{\mathbb{N}}\}$ with
$\Lambda_{n}=\frac{1}{n^4+1}Id_{\mathbb{C}}$. Since
$\Lambda_1=\frac{1}{2}Id_{\mathbb{C}}$ and
$\Lambda_2=\frac{1}{17}Id_{\mathbb{C}}$, the $g$-frame $\Lambda$ has
not any representation.
\end{exa}
By generalizing a result of the \cite{main}, the following theorem give sufficient conditions for $g$-frame 
$\Lambda=\{\Lambda_i\in B(\h,\K):i\in\mathbb{N}\}$ to have a representation. 
\begin{thm}\label{MT}
Let $\Lambda=\{\Lambda_{i}\in{B(\h,\K)}: i\in{\mathbb{N}}\}$ be a
$g$-frame that for any finite set $I_{n}\subset{\mathbb{N}}$, and
$\{g_{i}\}_{i\in{I_{n}}}\subset\K$,
$\sum_{i\in{I_{n}}}\Lambda_{i}^{*}g_{i}=0$ implies $g_{i}=0$ for
any $i\in{I_{n}}$. Suppose that $\ker T_{\Lambda}$ is
invariant under the right-shift operator $\mathcal{T}$. Then
$\Lambda$ is represented by  $T\in{B(\h)}$, where $\|T\|\leq\sqrt{B_{\Lambda}A_{\Lambda}^{-1}}$.
\end{thm}

\begin{proof}
Let $\{e_{i}\}_{i\in{I}}$ be an orthonormal basis for $\K$. We define the linear map 
$S: \Span\{\Lambda_{i}^*(\K)\}_{i\in{\mathbb{N}}}\rightarrow\Span\{\Lambda_{i}^*(\K)\}_{i\in{\mathbb{N}}}$ with 
\begin{align*}
S(\Lambda_{i}^* e_{j})=\Lambda_{i+1}^* e_{j}.
\end{align*}
For any finite index sets $I_n\subset{\mathbb{N}}$ and $J_{m}\subset
I$,
$\sum_{i\in{I_n},j\in{J_m}}c_{ij}\Lambda_{i}^*e_j=\sum_{i\in{I_n}}\Lambda_{i}^*\big(\sum_{j\in{J_m}}c_{ij}e_j\big)=0$
implies $c_{ij}=0$ for $i\in{I_n}$, $j\in{J_m}$. Therefore, $S$ is
well-defineded. Now, we show that $S$ is bounded. Let
$f=\sum_{i\in{I_n},j\in{J_m}}c_{ij}\Lambda_{i}^*e_j$ for
$c_{ij}\in\ell^2(\mathbb{C},\mathbb{N})$ with $c_{ij}=0$, $i\notin I_n$ or
$j\notin J_m$.\\
By the Theorem \ref{framegframe}, $F=\{\Lambda_i^*
e_j\}_{i\in{\mathbb{N}},j\in{J}}$ is a frame for $\h$ with lower and
upper frame bounds $A_{\Lambda}$ and $B_{\Lambda}$, respectively. We
can write $c_{ij}=d_{ij}+r_{ij}$ with $d_{ij}\in{\ker T_F}$ and
$r_{ij}\in{\Ran T_{F}^*}$. From
$\sum_{i\in{I_{n}}}\Lambda_{i}^*\big(\sum_{j\in{J_{m}}}d_{ij}e_j\big)=\sum_{i\in{I_n},j\in{J_m}}d_{ij}\Lambda_i^*e_j=0$,
we conclude that
$\sum_{i\in{I_n},j\in{J_m}}d_{ij}\Lambda_{i+1}^*e_j=0$. The same as
the proof of \cite{chretal}, we have
\begin{align*}
\|Sf\|^2=\big\|\sum_{i\in{I_n},j\in{J_m}}c_{ij}\Lambda_{i+1}^*e_j\big\|^2
&=\big\|\sum_{i\in{I_n},j\in{J_m}}r_{ij}\Lambda_{i+1}^*e_j\big\|^2
\\&\leq B_{\Lambda}\sum_{i\in{I_n},j\in{J_m}}|r_{ij}|^2.
\end{align*}
Since $\{r_{ij}\}_{i\in{I_n},j\in{J_m}}\in{(\ker T_{\Lambda})^{\perp}}$, by the \cite[Lemma 5.5.5]{c4} we have
\begin{align*}
A_{\Lambda}\sum_{i\in{I_n},j\in{J_m}}|r_{ij}|^2\leq\|\sum_{i\in{I_n},j\in{J_m}}r_{ij}\Lambda_i^*e_j\|^2.
\end{align*}
Therefore
\begin{align*}
\|Sf\|^2\leq B_{\Lambda}A_{\Lambda}^{-1}\big\|\sum_{i\in{I_n},j\in{J_m}}r_{ij}\Lambda_i^*e_j\big\|^2
&=B_{\Lambda}A_{\Lambda}^{-1}\big\|\sum_{i\in{I_n},j\in{J_m}}(d_{ij}+r_{ij})\Lambda_i^*e_j\big\|^2
\\&=B_{\Lambda}A_{\Lambda}^{-1}\big\|\sum_{i\in{I_n},j\in{J_m}}c_{ij}\Lambda_i^*e_j\big\|^2
\\&=B_{\Lambda}A_{\Lambda}^{-1}\|f\|^2.
\end{align*}
So, $S$ is bounded and can be extended to $\bar{S}\in{B(\h)}$. It is
obvious that $\Lambda$ is represented by $T=\bar{S}^*$ and $\|T\|\leq\sqrt{B_\Lambda A_\Lambda^{-1}}$.
\end{proof}
\begin{cor}\label{oth}
Every $g$-orthonormal basis has a representation.
\end{cor}

\begin{proof}
For every finite sequence $\{g_i\}_{i\in{I_n}}\subset \K$, we have
\begin{align*}
\big\|\sum_{i\in{I_n}}\Lambda_i^*g_i\big\|^2
=\big\langle\sum_{i\in{I_n}}\Lambda_i^*g_i,\sum_{j\in{I_n}}\Lambda_j^*g_j\big\rangle
&=\sum_{i\in{I_n}}\sum_{j\in{I_n}}\langle\Lambda_i^*g_i,\Lambda_j^*g_j\rangle
\\&=\sum_{i\in{I_n}}\langle g_i,g_i\rangle=\sum_{i\in{I_n}}\|g_i\|^2.
\end{align*}
So $\sum_{i\in{I_n}}\Lambda_i^*g_i=0$ implies $g_i=0$ for any
$i\in{I_n}$. Similarly, we have $\ker T_{\Lambda}=\{0\}$, that is invariant under
right-shift operator. Then, by the Theorem \ref{MT} the proof is
completed.
\end{proof}
\begin{rem}\label{gl}
Consider a $g$-frame $\Lambda=\{\Lambda_{i}\in B(\h,\K):i\in \mathbb{N}\}$ that is represented by $T$. For $S\in GL(\h)$, the family 
$\Lambda S=\{\Lambda_{i} S\in B(\h,\K):i\in \mathbb{N}\}$ is a $g$-frame \cite[Corollary 2.26]{AnajFR},, that is represented by $S^{-1}TS$.
\end{rem}  
\begin{cor}\label{grb}
Every $g$-Riesz basis has a representation. 
\end{cor}

\begin{proof}
By the Theorem \ref{reisoth}, Proposition \ref{oth} and Remark \ref{gl}, the proof is completed.
\end{proof}
Now, we give an example to show that the converse of the Theorem \ref{MT} is not satisfied.
\begin{exa}\label{e1}
Consider the tight $g$-frame $\Lambda=\{\Lambda_{i}\in B\big(l^2(\h,\mathbb{N})\big):i\in \mathbb{N}\}$ with $\Lambda_i=(\frac{1}{2})^{i-1}Id_{l^2(\h,\mathbb{N})}$. It is obvious that $\Lambda$ is represented by $\frac{1}{2}Id_{l^2(\h,\mathbb{N})}$, but 
$\Lambda_1^*(\frac{1}{2}e_1)+\Lambda_2^*(-e_1)=0$ for $e_1=(1,0,0,...)$.
\end{exa}
\begin{prop}\label{inv}
Let the $g$-frame $\Lambda=\{\Lambda_{i}\in B(\h,\K):i\in \mathbb{N}\}$ be represented by $T$. Then $\ker T_\Lambda$ is invariant under right-shift $\mathcal{T}$.
\end{prop}

\begin{proof}
For any $\{g_i\}_{i\in\mathbb{N}}\in\ker T_\Lambda$, we have 
\begin{align*}
T_\Lambda\mathcal{T}\{g_i\}_{i\in\mathbb{N}}
=\sum_{i\in\mathbb{N}}\Lambda_{i+1}^*g_i
=\sum_{i\in\mathbb{N}}T^*\Lambda_i^*g_i
=T^*\Big(\sum_{i\in\mathbb{N}}\Lambda_i^*g_i\Big)=0.
\end{align*}
\end{proof}
The following Proposition shows that the converse of the Theorem \ref{MT} is satisfied for one-dimentional Hilbert space $\K$.
\begin{prop}
Let $\K$ be a one-dimentional Hilber space and $\Lambda=\{\Lambda_{i}\in B(\h,\K):i\in \mathbb{N}\}$ be represented by $T$. 
If for finite index set $I_n\subset\mathbb{N}$ and $\{g_i\}_{i\in I_n}\subset \K$, we have $\sum_{i\in I_n}\Lambda_i^*g_i=0,$ 
then $g_i=0$ for any $i\in I_n$.
\end{prop}

\begin{proof}
Let $\{e_1\}$ be a basis fo $\K$. By the Theorem \ref{framegframe}, the sequence $F=\{\Lambda_i^*e_1\}_{i\in \mathbb{N}}=\{(T^*)^{i-1}\Lambda_1^*e_1\}_{i\in \mathbb{N}}$ is a frame for $\h$, and so by the Proposition \ref{mnref}, $F$ is linearly independent. We have
\begin{align*}
0=\sum_{i\in I_n}\Lambda_i^*g_i
=\sum_{i\in I_n}\Lambda_i^*(\alpha_i e_1)
=\sum_{i\in I_n}\alpha_i\Lambda_i^*e_1,
\end{align*}
therefore, for any $i\in I_n$, $\alpha_i=0$ and so $g_i=0$.
\end{proof}
\begin{rem}
The above Proposition shows that for one dimentional Hilbert space $\K$ with basis $\{e_1\}$, when a $g$-frame $\Lambda=\{\Lambda_{i}\in B(\h,\K):i\in \mathbb{N}\}$ has a representation, then the frame $\{\Lambda_i^*e_1\}_{i\in \mathbb{N}}$ has a representation. For finite dimentional Hilbert space $\K$ with orthonormal basis $\{e_j\}_{j=1}^n$, when a $g$-frame $\Lambda=\{\Lambda_{i}\in B(\h,\K):i\in \mathbb{N}\}$
is represented by $T$, then the frame $F=\{\Lambda_i^*e_j, j=1,..., n\}_{i\in \mathbb{N}}$ can be represented by $T^*$ and finite vectors $\{\Lambda_1^*e_1, ... , \Lambda_1^*e_n\}$, i.e., $F=\big\{(T^*)^{i-1}\Lambda_1^*e_j, j=1,..., n\big\}_{i\in \mathbb{N}}$, then, it can be worked on $g$-frames that be represented with finite family of operators.
But, the Example \ref{e1} shows that for infinite dimentional Hilbert space $K=l^2(\h,\mathbb{N})$ with orthonormal basis $\{e_j\}_{j\in I}$, it does not happen, i.e., a $g$-frame $\Lambda$ has a representation and the frame $\{\Lambda_i^*e_j\}_{i,j\in \mathbb{N}}$ does not have. Note that for a $g$-Reisz basis 
$\Lambda=\{\Lambda_{i}\in B(\h,\K):i\in \mathbb{N}\}$, the sequence $F=\{\Lambda_i^*e_j\}_{i\in\mathbb{N},j\in I}$ is a Riesz basis \cite{ws}. By the Corollary \ref{grb} and \cite[Example 2.2]{chretal}, both of the $\Lambda$ and $F$ have representations. What is the relation between these two representations (open problem)?
\end{rem} 
Now, we want to discuss the concept of representation for $g$-frames with index set $\mathbb{Z}$.
\begin{defn}
We say that a $g$-frame $\Lambda=\{\Lambda_i\in B(\h,\K):i\in\mathbb{Z}\}$ has a representation, if there is a $T\in{GL(\h)}$ such that $\Lambda_{i}=\Lambda_0 T^i,i\in{\mathbb{Z}}$. In the affrimative case, we say that $\Lambda$ is
represented by $T$.
\end{defn}
\begin{exa}
Consider the tight $g$-frame
$\Lambda=\{\Lambda_{n}\in{B(\mathbb{C})}: n\in{\mathbb{Z}}\}$ with
$\Lambda_{n}=\frac{1}{n^2-2n+4}Id_{\mathbb{C}}$. Since
$\Lambda_1=\frac{1}{3}Id_{\mathbb{C}}$ and
$\Lambda_3=\frac{1}{7}Id_{\mathbb{C}}$, the $g$-frame $\Lambda$ has
not any representation.
\end{exa}
\begin{rem}
Note that, all results that we have investigated for $g$-frames with index set $\mathbb{N}$, are satisfied for $g$-frames with index set $\mathbb{Z}$, as well.
\end{rem}
\begin{thm}
Let a $g$-frame $\Lambda=\{\Lambda_i\in B(\h,\K):i\in\mathbb{Z}\}$ is represented by $T$, then $\ker T_\Lambda$ is invariant under right-shift and left-shift and 
\begin{align*}
1\leq\|T\|\leq\sqrt{B_\Lambda A_\Lambda^{-1}},\quad 1\leq\|T^{-1}\|\leq\sqrt{B_\Lambda A_\Lambda^{-1}}.
\end{align*}
\end{thm}

\begin{proof}
Similar to the Proposition \ref{inv}, $\ker T_\Lambda$ is invariant under right-shift. For $\{g_i\}_{i\in \mathbb{Z}}\in\ker T_\Lambda,$
\begin{align*}
T_\Lambda\mathcal{T}^*\{g_i\}_{i\in \mathbb{Z}}
=\sum_{i\in \mathbb{Z}}\Lambda_{i-1}^*g_i
=\sum_{i\in \mathbb{Z}}(T^{i-1})^*\Lambda_0^*g_i
&=(T^{-1})^*\Big(\sum_{i\in \mathbb{Z}}(T^{i})^*\Lambda_0^*g_i\Big)
\\&=(T^{-1})^*\Big(\sum_{i\in \mathbb{Z}}\Lambda_i^*g_i\Big)
\\&=(T^{-1})^*T_\Lambda\{g_i\}_{i\in \mathbb{Z}}=0.
\end{align*}
So, $\ker T_\Lambda$ is also invariant under left-shift. Now for some fixed $n\in\mathbb{N}$ and $0\neq f\in\h$
we have 
\begin{align*}
A_\Lambda\|f\|^2\leq\sum_{i\in \mathbb{Z}}\|\Lambda_if\|^2
=\sum_{i\in \mathbb{Z}}\|\Lambda_0 T^if\|^2
&=\sum_{i\in \mathbb{Z}}\|\Lambda_0T^iT^{-n}T^nf\|^2
\\&=\sum_{i\in \mathbb{Z}}\|\Lambda_0T^{i-n}T^nf\|^2
\\&=\sum_{i\in \mathbb{Z}}\|\Lambda_iT^nf\|^2
\\&\leq B_\Lambda\|T^nf\|^2
\leq B_\Lambda\|T\|^{2n}\|f\|^2,
\end{align*}
that implies $\|T\|\geq1$. Since for any $i\in\mathbb{Z}$, $\Lambda_iT=\Lambda_{i+1}$, we have $T^*\Lambda_i^*e_j=\Lambda_{i+1}^*e_j$. So, $T^*$ is the operator $\bar{S}$ that is defined in the proof of the
Theorem \ref{MT}, just on $\Span\{\Lambda_i^*(\K)\}_{i\in\mathbb{Z}}$, and therefore we have $\|T\|\leq\sqrt{B_\Lambda A_\Lambda^{-1}}$, alike. Since 
$\Lambda=\{\Lambda_{-i}\in B(\h,\K):i\in\mathbb{Z}\}=\{\Lambda_0(T^{-1})^i:i\in\mathbb{Z}\}$, by 
replacing $T^{-1}$ instead of $T$, we get $1\leq\|T^{-1}\|\leq\sqrt{B_\lambda A_\Lambda^{-1}}.$
\end{proof}
Part (ii) of the Example \ref{tless1} shows that for the index set $\mathbb{N},$ $1\leq\|T\|$ does not happen, in general.
\begin{cor}
Let a $g$-frame $\Lambda=\{\Lambda_i\in B(\h,\K):i\in\mathbb{Z}\}$ is represented by $T\in GL(\h)$. Then the following hold:
\begin{enumerate}
\item[(i)] If $\Lambda$ is a tight $g$-frame, then $\|T\|=\|T^{-1}\|=1$, and so $T$ is isometry.
\item[(ii)] $\|S_{\Lambda}^{\frac{1}{2}}TS_{\Lambda}^{-\frac{1}{2}}\|=\|S_{\Lambda}^{\frac{1}{2}}T^{-1}S_{\Lambda}^{\frac{-1}{2}}\|=1$.
\end{enumerate}
\end{cor}
Authors of the paper \cite{philip} have considered sequences in $\h$ of the form $F=\{T^i f_0\}_{i\in I}$ , with
a linear operator $T$ to study for which bounded operator $T$ and vector $f_0\in\h$, $F$ is a frame for $\h$.
In \cite[Proposition 3.5]{chretal}, it was proved that if the operator $T\in B(\h)$ is
compact, then the sequence $\{T^i f_0\}_{i\in I}$ can not be a frame for infinte dimensional $\h$.
Someone can study these results for family of operators $\{\Lambda_0T^i\in B(\h,\K):i\in\mathbb{Z}\}$ for $T\in B(\h)$ and $\Lambda_0\in B(\h,\K)$.
\section{{\textbf Duality}}
The purpose of this section is to get a necessary and sufficient condition for a $g$-frame 
$\Lambda=\{\Lambda_i\in B(\h,\K_i):i\in\mathbb{N}\}$ to have a representation, by applying the concept of duality. Also, for some $g$-frames with representation we get a dual with representation and in one case without representation. At the end, we get the ralation between representations of dual $g$-frames by index set $\mathbb{Z}$. The proofs of the results are similar to \cite{main, chretal}. 
\begin{thm}\label{dual}
A $g$-frame $\Lambda=\{\Lambda_i\in B(\h,\K):i\in\mathbb{N}\}$ is represnted by $T$ if and only if for a dual 
$\Theta=\{\Theta_{i}\in B(\h,\K):i\in \mathbb{N}\}$ of $\Lambda$ (and hence all), 
\begin{align*}
 \Lambda_{k+1}=\sum_{i\in\mathbb{N}}\Lambda_k\Theta_i^*\Lambda_{i+1}.
\end{align*}
\end{thm}

\begin{proof}
First, assume that $\Lambda$ is represented by $T$. For any $g\in\K$ we have 
\begin{align*}
 T^*\Lambda_k^*g=T^*\Big(\sum_{i\in\mathbb{N}}\Lambda_i^*\Theta_i\Lambda_k^*g\Big)
 &=\sum_{i\in\mathbb{N}}T^*\Lambda_i^*\Theta_i\Lambda_k^*g
 \\&=\sum_{i\in\mathbb{N}}\Lambda_{i+1}^*\Theta_i\Lambda_k^*g
 =\sum_{i\in\mathbb{N}}\big(\Lambda_k \Theta_i^*\Lambda_{i+1}\big)^*g,
\end{align*}
then, $\Lambda_{k+1}=\sum_{i\in\mathbb{N}}\Lambda_k \Theta_i^*\Lambda_{i+1}$. 
\\Conversely, it is obvious that $\Lambda_iT=\Lambda_{i+1}$ for $Tf=\sum_{i\in\mathbb{N}}\Theta_i^*\Lambda_{i+1}f.$
\end{proof}
\begin{rem}
By \cite[Corollary 3.3]{ws}, for a $g$-Riesz basis $\Lambda=\{\Lambda_{i}\in B(\h,\K):i\in \mathbb{N}\}$, we have
\begin{align*}
 \Big\langle\sum_{i\in\mathbb{N}}\Lambda_k\widetilde{\Lambda}_i^*\Lambda_{i+1}f,g\Big\rangle
 &=\sum_{i\in\mathbb{N}}\langle\widetilde{\Lambda}_i^*\Lambda_{i+1}f,\Lambda_k^*g\rangle
\\ &=\sum_{i\in\mathbb{N}}\delta_{i,k}\langle\Lambda_{i+1}f,g\rangle
 =\langle\Lambda_{k+1}f,g\rangle,\quad f\in\h,g\in\K,
\end{align*}
therefore, by the Theorem \ref{dual}, $\Lambda$ has a representation. 
\end{rem}
In the following, we want to investigate that if a $g$-frame $\Lambda$ has a representation, its duals have representations or not.
If so, what is the relation between their representations?
\begin{exa}
\begin{enumerate}
\item[(i)] Assume that a $g$-frame $\Lambda=\{\Lambda_{i}\in B(\h,\K):i\in \mathbb{N}\}$ is represented by $T$. 
Then, by the Remark \ref{gl}, the canonical dual $\widetilde{\Lambda}$ is represented by $S_\Lambda T S_\Lambda^{-1}$.
\item[(ii)] Consider the $g$-frame $\Lambda=\{\Lambda_{i}\in B(\h):i\in \mathbb{N}\}$ with
 $\Lambda_i=(\frac{2}{3})^iId_{\h}$, that is represented by $\frac{2}{3}Id_{\h}$. The $g$-frame 
 $\Theta=\{\Theta_{i}\in B(\h,\K):i\in \mathbb{N}\}$ with $\Theta_i=(\frac{3}{4})^iId_{\h}$ is a dual of $\Lambda$ that is represented by $\frac{3}{4}Id_{\h}$. 
\item[(iii)] The $g$-frame $\Lambda=\{\Lambda_{i}\in B(\mathbb{C}):i=1,2,3\}$ with $\Lambda_i=2^{i-1}Id_{\mathbb{C}}$
is represented by $2Id_{\mathbb{C}}$, but the dual $\Theta=\{\Theta_{i}\in B(\mathbb{C}):i=1,2,3\}$ of $\Lambda$
with $\Theta_1=-2Id_{\mathbb{C}}, \Theta_2=Id_{\mathbb{C}}$ and $\Theta_3=\frac{1}{4}Id_{\mathbb{C}}$ does not 
have any representation. Note that the dual $\Gamma=\{\Gamma_{i}\in B(\mathbb{C}):i=1,2,3\}$ of $\Lambda$
with $\Gamma_i=\frac{1}{3}(\frac{1}{2})^{i-1}Id_{\mathbb{C}}$ is represented by $\frac{1}{2}Id_{\mathbb{C}}$.
\end{enumerate}
\end{exa}
\begin{prop}
Let a $g$-frame $\Lambda=\{\Lambda_i\in B(\h,\K):i\in\mathbb{Z}\}$ is represented by $T\in GL(\h)$. Then, the canonical dual $\widetilde{\Lambda}$ is represented by $S_\Lambda TS_\Lambda^{-1}=(T^*)^{-1}$.
\end{prop}

\begin{proof}
It is obvious that $\widetilde{\Lambda}$ is represented by $S_\Lambda TS_\Lambda^{-1}$. For any $\{g_i\}_{i\in\mathbb{Z}}\in l^2(\K,\mathbb{Z})$,
\begin{align*}
T^*T_\Lambda\{g_i\}_{i\in\mathbb{Z}}
=\sum_{i\in\mathbb{Z}}T^*\Lambda_i^*g_i
&=\sum_{i\in\mathbb{Z}}(\Lambda_i T)^*g_i
\\&=\sum_{i\in\mathbb{Z}}\Lambda_{i+1}^*g_i
=T_\Lambda\mathcal{T}\{g_i\}_{i\in\mathbb{Z}},
\end{align*}
So, we have 
\begin{align*}
T^*S_\Lambda T=T^*T_\Lambda T_\Lambda^* T
=T^*T_\Lambda(T^*T_\Lambda)^*=T_\Lambda\mathcal{T}\mathcal{T}^*T_\Lambda^*
=T_\Lambda T_\Lambda^*=S_\Lambda.
\end{align*}
Therefore, $S_\Lambda TS_\Lambda^{-1}=(T^*)^{-1}$.
\end{proof}
\begin{rem}
Let $F=\{f_i\}_{i\in\mathbb{Z}}$ and $G=\{g_i\}_{i\in\mathbb{Z}}$ be dual frames that is represented by $T,S\in GL(\h)$, respectively. Then, by the Remark \ref{frem}, the dual $g$-frames $\Lambda=\{\Lambda_{i}\in B(\h,\mathbb{C}):i\in\mathbb{Z}\}$ with $\Lambda_i f=\langle f,f_i\rangle$ and 
$\Theta=\{\Theta_{i}\in B(\h,\mathbb{C}):i\in\mathbb{Z}\}$ with $\Theta_i f=\langle f,g_i\rangle$ are represented by 
$T^*,S^*\in GL(\h)$, respectively. By the \cite[Lemma 3.3]{main}, $S=(T^*)^{-1}.$
\end{rem}
The relation between representations of dual $g$-frames by index set $\mathbb{Z}$ is given in below. 
\begin{thm}\label{dualrep}
Assume that $\Lambda=\{\Lambda_i\in B(\h,\K):i\in\mathbb{Z}\}=\{\Lambda_0T^i:i\in\mathbb{Z}\}$
and $\Theta=\{\Theta_{i}\in B(\h,\K):i\in\mathbb{Z}\}=\{\Theta_0S^{i}:i\in\mathbb{Z}\}$ are dual 
$g$-frames, where $T,S\in GL(\h)$ . Then, $S=(T^*)^{-1}$.
\end{thm}

\begin{proof}
For any $f\in\h$, we have 
\begin{align*}
f=\sum_{i\in\mathbb{Z}}\Lambda_i^*\Theta_i f
=\sum_{i\in\mathbb{Z}}(T^*)^i\Lambda_0^*\Theta_0S^i f
&=T^*\Big(\sum_{i\in\mathbb{Z}}(T^*)^{i-1}\Lambda_0^*\Theta_0S^{i-1}\Big)Sf
\\&=T^*\Big(\sum_{i\in\mathbb{Z}}\Lambda_i^*\Theta_i\Big)Sf
=T^*Sf.
\end{align*}
Since $T\in GL(\h)$, the proof is completed.
\end{proof}
In general, the Theorem \ref{dualrep} is not satisfied for index set $\mathbb{N}$ \big(see the Example \ref{dual}, (ii)\big).
$$\textbf{Acknowledgment:}$$

\end{document}